\documentclass[11pt,reqno,oneside]{amsart}
\usepackage{color} 
\usepackage{amssymb,amsmath,mathrsfs,amsthm}
\usepackage{caption}
\usepackage{enumitem}
\usepackage{mathtools}
\usepackage{geometry}
\usepackage{amsfonts}
\usepackage{color}

\usepackage{esint,comment}
\usepackage{apptools}
\AtAppendix{\counterwithin{thm}{section}}

\usepackage[dvips]{graphicx}
\allowdisplaybreaks
\newtheorem{thm}{Theorem}

 \newtheorem{lemma}[thm]{Lemma}

\theoremstyle{definition}

\def \no#1#2#3 {{\bf #1} (#3), #2.}
\def \eds#1#2#3 {#1, #2, #3.}

\def\R{{\mathbb R}}

\def\d{{\rm d}}

\def\:{{\colon}}

\def\be#1{\begin{equation}\label{#1}}
\def\ee{\end{equation}}

\def\<{\langle}
\def\>{\rangle}
\def\coloneqq{:=}

\newcommand{\lec}{\lesssim}
\newcommand{\bs}{\begin{split}}
\newcommand{\essss}{\end{split}}

\newcommand{\eqnb}{\begin{equation}}
\newcommand{\eqne}{\end{equation}}

\renewcommand{\ee}{\mathrm{e}}

\renewcommand{\R}{\mathbb{R}}

\newcommand{\wb}{\widetilde{b}}
\newcommand{\wG}{\widetilde{\Gamma}}
\newcommand{\wf}{\widetilde{f}}

\renewcommand{\d}{\mathrm{d}}

\begin{document}

\title[On minimal induced drag]{An improvement to Prandtl's 1933 model for minimizing induced drag} 
\author{W. S. O\.za\'nski}
\address{Florida State University, Department of Mathematics, Tallahassee, FL 32306, USA}
\email[]{wozanski@fsu.edu}
\maketitle

\date{}

\medskip

\begin{abstract}
We consider Prandtl's 1933 model for calculating circulation distribution function $\Gamma$ of a finite wing which minimizes induced drag, under the constraints of prescribed total lift and moment of inertia. We prove existence of a global minimizer of the problem without the restriction of nonnegativity $\Gamma\geq 0$ in an appropriate function space. We also consider an improved model, where the prescribed moment of inertia takes into account the bending moment due to the weight of the wing itself, which leads to a more efficient solution than Prandtl's 1933 result.  
\end{abstract}


\noindent\thanks{\em Keywords:\/}
vortex lines, aerodynamics, airfoil design, lifting-line theory, Prandtl, minimized induced drag, 3D airfoil

\section{Introduction}\label{sec_intro}

We consider the problem of minimal induced drag, which arises from Prandtl's lifting-line theory \cite{prandtl_trag,prandtl_NACA}. A few years after developing the fundamental principles of the theory, Prandtl published a short paper \cite{prandtl_33}, where he discussed a new solution to the drag minimization problem that is more efficient than the elliptic distribution. It appears that  only recently the findings of this paper have attracted some attention. For example, it has been studied extensively from the engineering perspective by NASA's Armstrong Flight Research Center \cite{bmjeg,newton}, see also the series of works on wing design by Hunsaker, Philips, and collaborators  \cite{philips_hunsaker_56b, philips_hunsaker_joo_56a, philips_hunsaker_taylor}. In this paper we aim to contribute a mathematical perspective.\\

 In order to introduce the problem, let us consider a symmetric wing of span $b>0$ and we let $\Gamma \colon (-b/2,b/2) \to \R$ denote the circulation distribution function. Namely, $\Gamma (x)$ denotes the circulation (in the $y,z$-plane) around a wing's crosssection at a spanwise location $x\in (-b/2,b/2)$. The downwash $w(x)$ is then given by 
\eqnb\label{downwash_intro}
w[\Gamma ](x) = \frac1{4\pi} \mathrm{p.v.} \int_{-b/2}^{b/2} \Gamma' (x' ) \frac{\d x'}{x-x'},  
\eqne
see \cite[(14)]{prandtl_betz} or \cite[(5.23)]{houghton}.

The induced drag on the wing can be then modelled by the expression 
\eqnb\label{def_of_drag}
D [\Gamma , b]\coloneqq \rho \int_{-b/2}^{b/2} \Gamma (x) w(x) \d x ,
\eqne
see \cite[(16)]{prandtl_betz}. We also refer the reader to \cite[Section~5.5]{houghton} for a modern exposition of the subject. By the Kutta-Joukowski law, the total lift on the wing is
\eqnb\label{A1}
A= \rho \nu \int_{-b/2}^{b/2} \Gamma (x) \d x,
\eqne
where $\rho$ denotes air density and $\nu>0$ denotes the freestream airspeed (see \cite{prandtl_33}), and the total moment of inertia is
\eqnb\label{A2}
 \rho \nu \int_{-b/2}^{b/2} |M(x)| \d x \leq Ar^2 ,
\eqne
where
\eqnb\label{bending_moment_intro}
M(x) \coloneqq  \int_x^{b/2} \Gamma (y ) (y-x)\d y
\eqne
is the bending moment acting at the location $x\in (0,b/2)$ of the wing due to the lift force acting on the segment between $x$ and $b/2$. 
Note that, assuming that 
\eqnb\label{moment_ass}
M (x) \geq 0\qquad \text{ for all }x\in (0,b/2),
\eqne  
one can substitute \eqref{bending_moment_intro} into \eqref{A2} and change the order of integration to obtain
\eqnb\label{A2_with_nonneg}
Ar^2 = \frac{\rho \nu}2 \int_{-b/2}^{b/2} \Gamma (x) x^2 \d x.
\eqne

The problem of the minimal induced drag is concerned with finding
\eqnb\label{MID_prob}
\inf_{b>0,\Gamma } D[\Gamma , b],
\eqne
where the infimum is taken over all $b$, $\Gamma$ such that \eqref{A1}, \eqref{A2} hold,  where $A$ and $Ar^2$ are given.

Clearly, the above problem is of central importance in aerodynamics.
 
\subsection{Prandtl approach}
In his 1933 paper \cite{prandtl_33} Prandtl considered downwash of the form
\eqnb\label{downwash_form} w(x) = C_1 + C_2 x^2
\eqne
for some $C_1,C_2\in \R$ and considered the dimension rescaling 
\[
\xi \coloneqq \frac{2x}b,
\]
so that $\xi \in (-1,1)$. He then noted that such downwash $w$ is achieved (via \eqref{downwash_intro}) if one chooses
\[
\Gamma (\xi ) = \Gamma_0 (1-\mu \xi^2 ) \sqrt{1-\xi^2 },
\]
where $\Gamma_0, \mu \in \R$ are parameters related to $C_1,C_2$ by 
\eqnb\label{choice_C12}
C_1 = \frac{\Gamma_0}{2b} \left( 1+\frac{\mu}{2} \right)  , \qquad C_2 = -\frac{6 \mu \Gamma_0}{b^3}.
\eqne
Substituting  this ansatz into \eqref{def_of_drag} gives
\eqnb\label{temp01}
D = \frac{\pi \rho \Gamma_0^2 }{8} \left( 1-\frac{\mu}{2} +\frac{\mu^2}4 \right),  
\eqne
and the constraints of the given total lift \eqref{A1} and the given total moment of inertia \eqref{A2} taken with ``$=$'' give\footnote{We note that there is a computation error in (9) in the Prandtl's paper, where ``$\sqrt{2}$'' is missing, this is a consequence of the fact that ``$64$'' in the constraint \cite[(8)]{prandtl_33} should be replaced by ``$128$''.}
\[
A=\rho \nu \frac{b\pi \Gamma_0}4 \left( 1-\frac{\mu}4 \right)\qquad \text{ and }\qquad b=4 \sqrt{2} r \sqrt{\frac{1-\frac{\mu}4}{1-\frac{\mu}2}}.
\]
respectively. Using the last formula (for $b$) in the previous one, and then substituting $\Gamma_0$ into  \eqref{temp01} gives 
\eqnb\label{prandtl_optim_drag}
D = \frac{A^2}{16\pi \rho \nu^2 r^2 } \frac{\left(1-\frac{\mu}2 \right) \left(1-\frac{\mu}2 + \frac{\mu^2}4 \right) }{\left( 1-\frac{\mu}4 \right)^3}.
\eqne
We can now observe that the right-hand side is decreasing for $\mu \in (0,2)$. However, the case $\mu >1$ becomes nonphysical as bending moments can become negative. Thus the minimum is achieved at $\mu=1$, which is the Prandtl 1933 solution. We also note that the corresponding induced drag $D = A^2/18 \pi \rho \nu^2 r^2$ equals to $8/9$ of the induced drag in the elliptic case\footnote{Note that taking $\mu =0$ in \eqref{choice_C12} gives circulation distribution $\Gamma (\xi ) = \Gamma_0 \sqrt{1-\xi^2}$, whose plot is an arc of an ellipse, hence ``elliptic case'' or ``elliptic distribution''.}  $\mu =0$, i.e., $11.11...\%$ reduction. 

\section{The aim of the paper}\label{sec_aim}

In this paper we address two concerns regarding the above argument. 

First we note that  the  ansatz \eqref{downwash_form} might appear arbitrary, and so, in particular, it is not clear whether such form of $w$ would yield a global minimizer (i.e., among all functions). To this end, in Section~\ref{sec_minimizer} we prove existence of a global minimizer of the problem \eqref{MID_prob} in an appropriate function space.

Secondly,  we make an observation that the model considered above can be improved by considering the bending moment due to the weight of the wing itself (in \eqref{bending_moment_intro}). In fact, such a moment acts in the opposite direction to the lift force, which shows that one can satisfy the moment constraint with larger $\Gamma$, which in turn leads to a solution with smaller induced drag. We thus show, in Section~\ref{sec_new_model},  that the corresponding minimization problem (see Theorem~\ref{T02} below) also  admits a minimizer (in the same function space), and we give an example of $\Gamma$ that is more efficient than Prandtl's solution \eqref{prandtl_optim_drag} above (with $\mu=1$).  

For simplicity, we  assume that all material constants $A,r,\rho,\nu$ equal $1$.  

\section{Existence of a minimizer of the Prandtl's problem}\label{sec_minimizer}

Here we show that the Prandtl problem \eqref{MID_prob} of the minimal induced drag admits a global minimizer; see Theorem~\ref{T01}. \\

In order to motivate the result, we first recall the work of Hunsaker and Philips \cite{philips_hunsaker_56b} who treat the problem using the Fourier Sine Series and show that only the first two Fourier modes are used to guarantee that the constraints \eqref{A1},\eqref{A2_with_nonneg} hold. In order to illustrate this, we will use the approach of Chebyshev polynomials. Namely we consider $\Gamma (\xi )$ of the form
\[
\Gamma (\xi ) = \left( a_0 U_0 (\xi ) + a_2 U_2 (\xi )+ \ldots \right) \sqrt{{1-\xi^2}},
\]
where $U_{l}$ denotes the Chebyshev polynomial of the second kind of order $l$ (see \eqref{cheb1}--\eqref{cheb2}). We  only use even orders $l$ since $\Gamma$ is assumed to be even. We discuss the Chebyshev polynomials in Appendix~\ref{ap_chebyshev}, but we recall at this point that
\[
U_0 (\xi ) =1,\qquad U_2 (\xi )= 4\xi^2 -1.
\]
Also, note that the set of Chebyshev polynomials $\{ U_{2k} \}$ forms an orthonormal basis for
\eqnb\label{def_H}
H \coloneqq \{ f\colon (0,1) \to \R \colon \| f \|_{H}<\infty \},
\eqne
equipped with the inner product
\[
(f,g) \coloneqq \int_0^1 f(\xi ) g(\xi )\sqrt{1-\xi^2} \d \xi \qquad \text{ and norm }\quad \| f \|_{H} \coloneqq (f,f)^{1/2}.
\]
 We note that if 
\[
\Gamma (\xi ) = f(\xi) \sqrt{1-\xi^2} \qquad \text{ for some }f\in H
\]
then the total lift constraint \eqref{A1} becomes
\eqnb\label{A1_simple}
a_0 = \frac{4}{\pi b},
\eqne
due to the orthogonality of the $U_n$'s, namely \eqref{cheb_u_ortho}.\\

We emphasize that the $U_{n}$'s are eigenfunctions of the downwash operator $w[\Gamma]$ composed with multiplication by $\sqrt{1-\xi^2}$, with eigenvalues $(n+1)/4$, i.e.,
\eqnb\label{eigen_U}
w[U_n \sqrt{1-(\cdot )^2}](\xi ) = \frac{n+1}4 U_n (\xi).
\eqne
Indeed, using the relations between the $U_n$'s and the Chebyshev polynomials of the first kind $T_n$ (see \eqref{cheb_der}, \eqref{cheb_integral1} in Appendix~\ref{ap_chebyshev}) we obtain
\[
\begin{split}
w[U_n \sqrt{1-(\cdot )^2}](\xi ) &= \frac1{4\pi} \mathrm{p.v.}\int_{-1}^1 \frac{(U_n (\eta ) \sqrt{1-\eta^2})'}{\xi-\eta }\d \eta \\
&= \frac1{4\pi} \mathrm{p.v.}\int_{-1}^1 \frac{U_n' (\eta ) (1-\eta^2) -\eta U_n(\eta)  }{\sqrt{1-\eta^2}(\xi-\eta)}\d y \\
&= -\frac{n+1}{4\pi} \mathrm{p.v.}\int_{-1}^1 \frac{T_{n+1} (\eta)  }{\sqrt{1-\eta^2}(\xi-\eta)}\d \eta \\
&= \frac{n+1}4 U_n (\xi),
\end{split}
\]
proving \eqref{eigen_U}.

Consequently, the induced drag $D[b,\Gamma ]$ has a simple expression for $\Gamma \in H$, which reads
\eqnb\label{drag_norm}
\begin{split}
D[b,\Gamma ] &= 2\int_{0}^1 \Gamma (x) w[\Gamma ] (x) \d x = 2\sum_{m,n\geq 0} a_{2n} a_{2m} \frac{2m+1}4 \int_0^1 U_{2n}(\xi ) U_{2m} (\xi ) \sqrt{1-\xi^2} \d \xi  \\
&=\frac{\pi}{8} \sum_{m\geq 0} (2m+1) a_{2m}^2 .
\end{split} 
\eqne
Here, $b$ is given by \eqref{A1_simple}, but one can see that $D[b,\Gamma ]$ is independent of $b$. In fact, for each $b$, $\sqrt{D[b,\cdot ]}$ can be thought of as the $H^{1/2}$ norm with respect to the orthonormal basis $\{ U_{2n} \}_{n\geq 0}$. 
In particular, in light of the constraints \eqref{A1}, \eqref{A2_with_nonneg} and the orthogonality \eqref{cheb_u_ortho} of the Chebyshev polynomials, we note that only the first two modes, $a_0,a_2$ are relevant to satisfy the constraints \eqref{A1}, \eqref{A2_with_nonneg}, while the remaining modes, if nontrivial, only increase $D$. Thus, let us restrict ourselves to the case where $a_{2k}=0$ for all $k\geq 2$. Then the problem reduces to \eqref{downwash_form}; namely, if $\Gamma (\xi ) = (a_0U_0 (\xi )+ a_2U_2 (\xi )) \sqrt{1-\xi^2} $ then $\Gamma\geq 0$ if and only if 
\eqnb\label{nonneg_G_a2a0}
a_2 \geq -a_0/3,
\eqne
see Fig.~\ref{fig_prandtl1}. In the borderline case $a_2=-a_0/3$ we thus obtain
\eqnb\label{01}
D = \frac{\pi}8 \left( a_0^2 + 3 a_2^2 \right)  = \frac{\pi }{6} a_0^2.
\eqne
Since \eqref{moment_ass} ($M\geq 0$) is assumed in the Prandtl solution, we can thus find $a_0$ from the moment constraint \eqref{A2_with_nonneg}, which becomes
\eqnb\label{calc1}\begin{split}
1&= \frac{b^3}{16} \int_{-1}^1 (a_0 U_0 (\xi ) + a_2 U_2(\xi ))\sqrt{1-\xi^2} \, \xi^2\, \d \xi \\
&= \frac{4}{\pi^3a_0^3} \int_{-1}^1 (a_0 U_0 (\xi ) + a_2 U_2(\xi ))\sqrt{1-\xi^2}  \frac{U_2(\xi ) + U_0 (\xi )}4 \d \xi \\
&= \frac{a_0 + a_2}{2\pi^2a_0^3} = \frac{1}{3\pi^2a_0^2}  
\end{split}
\eqne
due to \eqref{A1_simple} and the orthogonality  \eqref{cheb_u_ortho} of Chebyshev polynomials.
(Recall we assume that $A=r=\rho=\nu=1$ for simplicity.)
Substitution into \eqref{01} gives $D= \frac{1}{18\pi }$, which gives the same result as \eqref{prandtl_optim_drag} at the optimal point $\mu=1$, see also Figure~\ref{fig_prandtl1} below.
\begin{center}
 \includegraphics[width=9cm]{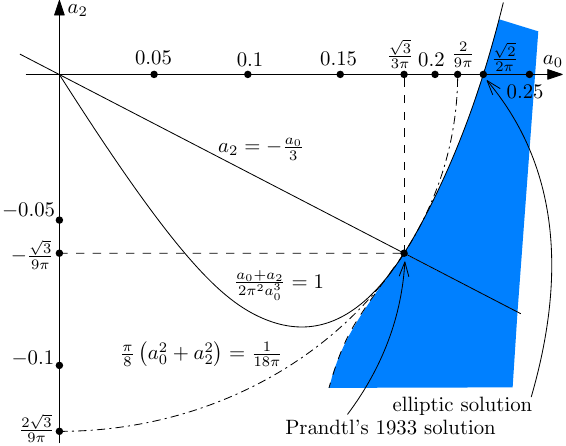}
 
 \nopagebreak
 \captionsetup{width=.8\linewidth}
  \captionof{figure}{A geometric interpretation of the Prandtl's 1933 solution to the minimimal induced drag problem. Note that the region that Prandtl \cite{prandtl_33} considered corresponds to the part of blue region above the line $a_2=-a_0/3$. The part below the line (which corresponds to the case when $\Gamma$ can change sign) was obtained by a numerical simulation of the constraint \eqref{A2}. }\label{fig_prandtl1} 
\end{center}

Note that, in this two-dimensional restriction, the minimizer of the induced drag problem is the point in the region determined by the nonnegativity constraint \eqref{moment_ass} and the moment constraint \eqref{A2_simple} (depicted in the sketch by the blue region) that is closest to the origin in the sense of ellipse $a_0^2+3a^2_2=\text{const.}$ of the smallest radius (depicted in the sketch by the dash-dotted line), recall~\eqref{01}. In particular, one can observe that the  elliptic distribution (i.e., when $\Gamma (\xi ) = a_0 \sqrt{1-\xi^2}$) admits minimal induced drag $D=1/16\pi$ at $a_0=\sqrt{2}{2\pi}$, which is $9/8$ times larger than the Prandtl solution, as mentioned below \eqref{prandtl_optim_drag}. 

We note that the above analysis is equivalent to the approach using the Fourier Sine series upon the change of variable $\xi = -\cos (\theta )$, which was used in \cite{philips_hunsaker_56b, philips_hunsaker_joo_56a,philips_hunsaker_taylor}.\\

We emphasize that the above analysis is carried out under the assumption that $\Gamma \geq 0$, which implies that $M\geq 0$, so that the change of order of integration in \eqref{A2_with_nonneg} is justified. Without this assumption, it is not clear why would the first two Chebyshev polynomials suffice in search of a global minimizer. In fact, one can imagine a scenario, when introducing higher order oscillation one can obtain a cancellation inside the integral defining $M$. 
To address this, we prove the existence of global minimizers of the minimal induced drag problem under only the assumptions \eqref{A1}--\eqref{A2}. 

\begin{thm}[Minimizer of the Prandtl problem]\label{T01}
Given $b_0>0$ there exist $\wb\in [0,b_0]$ and $\wG \in H$ satisfying the total lift and moment assumptions \eqref{A1}, \eqref{A2} such that 
\[
D[\wb ,\wG] = \min_{\substack{b\in [0,b_0],\\ \Gamma \in H\text{ satisfying }\eqref{A1},\eqref{A2}}} D[b,\Gamma ].
\]
\end{thm}

We note that Theorem~\ref{T01} does not assume nonnegativity of $M$ or $\Gamma$. We note, since the space of smooth functions is dense in $H$, the claim provided by Theorem~\ref{T01} characterizes the regularity of a global minimizer. 

The restriction for $b\in [0,b_0]$ arises from the possible degeneration when one looks for a general solution (in $H$, not just $H_2\coloneqq \mathrm{Span}\{ U_0, U_2 \}$), as it is not clear how to keep track of the moment constraint \eqref{A2}. Namely, one may think of a scenario where $\Gamma (\xi ) = \sqrt{1-\xi^2} \sum_{n\geq 0} a_{2n}U_{2n} (\xi)$ has many nonzero coefficients $a_{2n}$ which are chosen in a way that while $\Gamma(\xi )$ changes sign many times as $\xi $ moves from $0$ to $1$. In such a scenario, $M$ can be reduced (particularly for $\xi <1/2$) since the integral \eqref{bending_moment_intro} defining $M$ would admit many cancellations. At the same time it is perhaps possible to reduce the coefficients $a_{2n}$ for small $n$, so that the induced drag could shrink, since the moment of inertia constraint \eqref{A2} would become easier to satisfy. Namely, the distance to $0$, in the sense of \eqref{drag_norm}, of the region of  $\{ a_0, a_2 , \ldots \} \in l^2$ for which \eqref{A2} hold would shrink. Possibly, it could shrink to zero, so that one could obtain an infinitely long wing (i.e., $a_0=0$ in the limit) of zero induced drag. On the one hand, one would expect that such a degenerate scenario could be excluded, although we are unable to prove it at this point. On the other hand, introduction of higher coefficients $a_{2n}$ is reminiscent of the tubercles appearing on humpback whale flippers, which have been observed to delay the stall angle of a flipper, as well as reduce induced drag \cite{fish_battle_1995,MMHF_2004}.\\

Here we only observe that in the case of only the first two modes  present (i.e., $a_{2n}=0$ for $n\geq 2$), which does not allow oscillations of $\Gamma (\xi)$, numerical simulation suggests that the minimized induced drag is achieved in the case when $\Gamma $ is a nonnegative function. To be precise, a simple computation using MATLAB allows one to compute the value of the total moment 
\[
\frac{ b^3}4 \int_0^1 \left| \int_\xi^1 (a_0 U_0 (\eta ) + a_2 U_2 (\eta ) ) \sqrt{1-\eta^2} (\eta - \xi )\d \eta \right| \d \xi 
\] 
for various $a_0,a_2$, which we present in Figure~\ref{fig_prandtl1} (the part below the line $a_2=-a_0/3$). This part is separated away from the ellipse $\frac{\pi}8 (a_0^2+a_2^2)= \frac{1}{18\pi}$, which corresponds to the minimized induced drag $D$. \\

Before proving Theorem~\ref{T01} we note that  the main difficulty, which is also our obstacle to obtaining uniqueness, is the lack of convexity of the constraints \eqref{A1},\eqref{A2}. These are convex with respect to $\Gamma$, but not convex with respect to the pair $(b,\Gamma)$. Recall from Section~\ref{sec_aim} that we assume, for simplicity, we assume that  $\nu=\rho=1$;  we also apply the rescaling $x=\xi b/2$, and we  abuse the notation slightly by writing $\Gamma (\xi )=\Gamma (x)$. Then, considering $\Gamma$ of the general form 
\eqnb\label{convention}
\Gamma (\xi ) = f(\xi ) \sqrt{1-\xi^2} = (a_0U_0 (\xi) + a_2 U_2 (\xi )+\ldots ) \sqrt{1-\xi^2},
\eqne
suggested by the space $H$ (recall \eqref{def_H}), the total lift constraint \eqref{A1} becomes \eqref{A1_simple} and the total moment constraint \eqref{A2} becomes
\eqnb\label{A2_simple}
\frac{b^3}{4} \int_0^1 \left| \int_{\xi}^1 f(\eta) \sqrt{1-\eta^2} (\eta - \xi )\d \eta \right| \d \xi \leq 1.
\eqne

\begin{proof}[Proof of Theorem~\ref{T01}.]
Let $b_k\in [0,b_0]$, $\Gamma_k= f_k \sqrt{1-(\cdot )^2}$ with $f_k = a_{0,k}U_{0}+ a_{2,k} U_2 + \ldots  \in H$ be such that \eqref{A1}, \eqref{A2} hold and 
\[
D[b_k , \Gamma_k ] \longrightarrow \inf_{\substack{b\in [0,b_0],\\ f\in H \text{ satisfying }\eqref{A1},\eqref{A2}}} D[b,\Gamma ]
\]
as $k\to \infty$ (here $\Gamma = f\sqrt{1-(\cdot )^2}$). Since $\| f \|_H\lec D[b,\Gamma ]$ (recall \eqref{drag_norm}) the sequence $\{ f_k \}$ is bounded in $H$, and so, by weak compactness, there exists a subsequence (which we relabel back to $k\geq 0$) such that $f_k \to \wf$ weakly in $H$ as $k\to \infty$, for some $\wf = \widetilde{a_0}U_0 + \widetilde{a_2} U_2 + \ldots  \in H$. In particular
\[
a_{0,k} \coloneqq (f_k,U_0) \to (\wf , U_0 ) = \widetilde{a_0}.
\]
By the total lift constraint \eqref{A1_simple} we thus have $b_k \to \wb$ for some $\wb \in [0,b_0]$. 
It remains to show that the limit $(\wb , \wf )$ satisfies the total moment constraint \eqref{A2_simple}. We note that \eqref{A2_simple} is convex with respect to $f$ (and so is preserved under weak limits), but an additional factor of $b^3$ makes it nonconvex with respect to $( b,f)$. In order to fix this, we apply the classical characterization of convex sets in Banach spaces (see Appendix~\ref{ap_convex_sets}). Namely we set $A_k \colon H\to H$ as $A_k f \coloneqq \frac{b_k}8 f$, and similarly $\widetilde{A} f \coloneqq \frac{\wb}8 f$. Then $\| A_k - \widetilde{A} \|_{H\to H} \to 0$ as $k\to \infty$ and the moment constraint \eqref{A2_simple} for $(b_k,f_k)$ becomes 
\[
A_k f_k \in K,
\]
where $K\subset H$ is the convex set of all $f\in H$ satisfying \eqref{A2_simple} with $b=1$. Applying Lemma~\ref{L01} we thus obtain $\widetilde{A}\wf \in K$, as required.
\end{proof}


\section{An improved model}\label{sec_new_model}
Here we present a more comprehensive model which involves both the bending moment arising from the lift force and the bending moment due to the weight of the wing itself. This is particularly important for the restriction \eqref{A2}, where one  interprets $Ar^2$ as the total weight of the wing. In fact, one could obtain the total weight of a wing by considering instead
\[
Ar^2 C_w,
\]
where $C_w$ is a material constant that measures the resistance to bending moments. Namely $C_w$ represents the weight of the wing per unit span length (i.e., weight of the cross-section) that is needed to support $1Nm$ of bending moment. For brevity, let us assume that $C_w=1$, so that we can think of \eqref{A2} as the weight constraint. In order to incorporate the bending moment due to the weight of the wing itself into the this constraint, one should replace \eqref{bending_moment_intro} with
\eqnb\label{def_M_new}
M(x) = \rho \nu \int_x^{b/2} \Gamma (y ) (y-x) \d y - \rho_w C_w \int_x^{b/2} M(y) (y-x) \d y = \int_x^{b/2} (\Gamma (y) - M(y) ) (y-x)\d y,
\eqne
where, in the second equality, we have set all material constants $\rho_w, C_w$ and free airstream constants $\nu , \rho$ equal to $1$, for simplicity.

We note that this definition of $M$  is implicit, and so $M$ is not immediately well-defined (given $\Gamma$). However it can be observed that the unique solution $M$ is given by
\eqnb\label{sol}
M(\xi ) = \frac{b}{2} \int_\xi^1 \Gamma (\eta ) \sin \left(b (\eta - \xi )/2\right) \d \eta ,
\eqne
where we also abuse the notation to write $M(x) = M(\xi )$, for $x=\xi b/2$.
Indeed, that such $M$ satisfies \eqref{def_M_new} can be verified directly. On the other hand, if $\widetilde{M}$ is another solution of \eqref{def_M_new} then both $M$ and $\widetilde{M}$ satisfy the ODE
\[
M''+ \frac{b^2}4  M = \frac{b^2}4 \Gamma
\]
with the boundary conditions $M(1)=M'(1)=0$, and so $M=\widetilde{M}$, which gives uniqueness.\\

Except for \eqref{A2} we also impose additional assumption that 
\eqnb\label{noncol}
M(x) \geq 0 \qquad \text{ and }\qquad M(x) \geq \int_x^{b/2} M(y) (y-x) \d y\qquad \text{ for all }x\in [0,b/2],
\eqne
which models the constraint that the wing does not collapse due to the bending moment arising from its own weight.

We are thus interested in finding 
\eqnb\label{new_inf}
\inf_{\substack{b\in [0,b_0],\\ \Gamma \in H\text{ satisfying }\eqref{A1},\eqref{A2},\eqref{noncol}}} D[b,\Gamma ],
\eqne
where $M$ is given by \eqref{sol}. 
Analogously to Theorem~\ref{T01}, we can prove existence of a global minimizer.
\begin{thm}[Minimizer of the minimal induced drag problem \eqref{new_inf}]\label{T02}
Given $b_0>0$ there exist $\wb\in [0,b_0]$ and $\wG \in H$ satisfying the total lift and moment assumptions \eqref{A1}, \eqref{A2} with the moment function $M$ defined by \eqref{def_M_new}, such that 
\[
D[\wb ,\wG] = \min_{\substack{b\in [0,b_0],\\ \Gamma \in H\text{ satisfying }\eqref{A1},\eqref{A2},\eqref{noncol}}} D[b,\Gamma ].
\]
\end{thm}
\begin{proof}
The proof is analogous to the proof of Theorem~\ref{T01}. 
\end{proof}

\subsection{Solution of the improved model in $H_2$}\label{sec_h2}
Here we restrict ourselves to the case of the first two modes, i.e. we consider $\Gamma(\xi ) =  (a_0 U_0 (\xi ) + a_2 U_2 (\xi ))\sqrt{1-\xi^2}$, in which case  \eqref{sol} becomes
\[
M(\xi ) = \frac{b}2   \int_{\xi}^1(a_0 U_0 (\eta ) +a_2 U_2 (\eta )) \sqrt{1-\eta^2} \sin (b(\eta- \xi)/2)\d \eta  .
\]
We consider only nonnegative circulation distribution, $\Gamma \geq 0 $, i.e. $a_2\geq -a_0/3$. Then the  total lift constraint \eqref{A1} (as before) becomes \eqref{A1_simple}. On the other hand, the weight constraint \eqref{A2} becomes
\eqnb\label{A2_new}
\frac{b^2}4 \int_0^1 (a_0 U_0 (\xi ) + a_2 U_2 (\xi )) \sqrt{1-\xi^2 } \left( 1-\cos \frac{b\xi }2 \right) \d \xi= b \int_0^1 M(\xi ) \d \xi \leq 1
\eqne
(recall that we suppose that all constants equal $1$). In contrast to Prandtl's model (recall \eqref{calc1}), we  cannot rewrite \eqref{A2_new} as an explicit relation between $a_0 $ and $a_2$. Instead, we rely on a numerical simulation using MATLAB, and we present the result in Figure~\ref{fig_prandtl2} below.
\begin{center}
 \includegraphics[width=9cm]{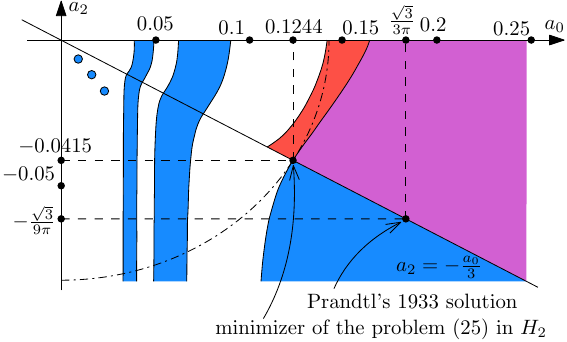}
 
 \nopagebreak
 \captionsetup{width=.8\linewidth}
  \captionof{figure}{A geometric interpretation of the optimal solution of the problem \eqref{new_inf}, and comparison to Prandtl's 1933 solution to the minimimal induced drag problem. Here the region above the line $a_2=-a_0/3$ represent the nonnegativity assumption $M\geq 0$ (see \eqref{noncol}), the blue region represents the total weight constraint \eqref{A2}, and the red region represents the ``noncollapse'' condition (i.e., the second part of \eqref{noncol}). The blue dots at the top-left corner indicate that the nonphysical vertical blue regions continue to appear and concentrate aa $a_0\to 0^+$.}\label{fig_prandtl2} 
\end{center}

Figure~\ref{fig_prandtl2} shows that the minimum of the induced drag under condition \eqref{noncol} and prescribed total lift \eqref{A1}, total weight \eqref{A2} corresponds to the point in the purple region (i.e. both red and blue) that is located on or above the line $a_2=-a_0/3$ and that is closest to the origin in the sense of minimized $a_0^2+3a_2^2$ (recall~\eqref{01}), represented in the figure by the dot-dashed line. Based on the sketch, the induced drag is about $0.45828$ of the induced drag corresponding to the Prandtl's solution \eqref{01} (i.e., obtained using \eqref{bending_moment_intro} as the definition of $M$, rather than \eqref{sol}), i.e. a $54.2\%$ reduction.

This seemingly high number is a result our assumption of all physical constants equal to $1$. For the given values of physical constants  the improvement is, most likely, much smaller, but the improved model will always give a smaller induced drag than Prandtl's 1933 solution. Indeed, the defintion \eqref{def_M_new} involves the lift force, arising from $\Gamma$, and the force due to weight, arising from $M$, which point in opposite directions, and so a cancellation of these forces, expressed in the solution \eqref{sol}, gives a more efficient weight distribution.

\section*{Discussion}

Let us first consider the hypothetical scenario in which one is able to produce infinitely strong materials. In that case, i.e., when $C_w \to \infty$, this would give arbitrarily small induced drag, since the total mass constraint, $Ar^2 C_w \geq \rho \nu \int_{-b/2}^{b/2} |M(x)| \d x$, becomes less and less relevant, and so, taking $a_{4}=a_6=\ldots =0$ and $a_2\to 0$, the induced drag \eqref{drag_norm} becomes $D=\pi a_0^2/8 \to 0$ by taking $a_0\to 0$. Then also the wingspan $b\sim 1/a_0\to \infty$, due to the lift constraint \eqref{A1}. This is reminiscent of supersailplanes, such as Eta or Concordia \cite{eta,concordia}. \\

We note that the moment $M$ definition \eqref{def_M_new} constraint admits easy generalizations, including spanwise material strength variability, or incorporation of additional, non-structural weight, such as engines or fuel tanks. For example, if $\rho(x)$ denotes a prescribed spanwise additional weight distribution, so that $\int_{a_0}^{a_1} \rho (x) \d x$ is the additional weight located between $x=a_0$ and $x=a_1$, then \eqref{def_M_new} should be augmented by subtracting $\rho(x)$ on the right-hand side, which leads to an ODE that can be solved similarly to \eqref{sol}.



\section*{Declarations}
\subsection*{Funding and/or Conflicts of interests/Competing interests}

There is no conflict of interests. There are no competing interests.

\appendix

\renewcommand{\theequation}{\thesection.\arabic{equation}}

\section{Chebyshev polynomials}\label{ap_chebyshev}
We recall \cite[Section~22]{abramowitz_stegun} Chebyshev polynomials $T_n$, $U_n$ of the first kind and second kind, respectively, defined by
\eqnb\label{cheb1}
T_0(x) = U_0 (x) \coloneqq 1, \quad T_1(x) \coloneqq x, \quad U_1 (x) \coloneqq 2x,
\eqne
and by
\eqnb\label{cheb2}
\begin{split} T_{n+1}  (x)& \coloneqq 2x T_n (x) -T_{n-1} (x),\\
U_{n+1} (x) &\coloneqq  2x U_n (x) - U_{n-1} (x)
\end{split}
\eqne
for $n\geq 1$. The Chebyshev polynomials are related to each other in a number of remarkable ways. Here we only list those properties used in the note. First, we have 
\eqnb\label{cheb_der}
U_n' = \frac{xU_n - (n+1) T_{n+1}}{1-x^2}
\eqne
for all $n$. Second,
\eqnb\label{cheb_integral1}
\int_{-1}^1 \frac{T_n (y) }{\sqrt{1-y^2} (x-y) } \d y = - \pi U_{n-1} (x)
\eqne
for all. Finally, the $U_n$'s form an orthonormal basis of $H$ (recall~\eqref{def_H}), with
\eqnb\label{cheb_u_ortho}
\int_{-1}^1 U_n (x) U_m(x) \sqrt{1-x^2} \d x = \begin{cases}
0\qquad &m\ne n,\\
\pi/2 & m=n.
\end{cases}
\eqne

\section{Convex sets in Banach spaces}\label{ap_convex_sets}
Here we prove the following fact relating convex sets in Banach spaces and weak convergence.
\begin{lemma}[Convex sets in Banach spaces and weak limits]\label{L01}
Let $X,Y$ be real Banach spaces. Suppose that $x_k \rightharpoonup x$ in $X$ as $k\to \infty$, and that $A_k,A \in B(X,Y)$ are such that $\| A_k - A \|_{B(X,Y)} \to 0$ as $k\to \infty$. Let $K\subset Y$ be a closed convex set. Then,
\[
\text{ if }A_kx_k \in K\text{ for all }k,\qquad \text{ then }\qquad Ax \in K.
\]
\end{lemma}
\begin{proof}
Let $l\in Y^*$ denote any supporting hyperplane of $K$ in $Y$, i.e. 
\[
K\subset \{ y\in Y \colon l(y)\leq l_0 \}
\]
for some $l_0\in R$. Since $\{ x_k \}$ converges weakly, it is bounded, and so there exists $M>0$ such that $\| x_k \|\leq M$ for all $k$. Given $\varepsilon >0$ we let $k$ be sufficiently large so that
\eqnb\label{temp02}
|l(A(x-x_k)|\leq \frac{\varepsilon }2 \qquad \text{ and }\qquad \| A-A_k \| \leq \frac{\varepsilon }{2\| l \|_{Y^*}M}.
\eqne
Note that such choice is possible by assumptions, since  $l\circ A \in X^*$. We obtain
\[
\begin{split}
l(Ax) &= l(A(x-x_k)) + l((A-A_k)x_k ) + l(A_k x_k ) \\
&\lec |l(A(x-x_k))| + \| l \|_{Y^*} \| A- A_k \| M + l_0 \\
&\leq l_0 + \varepsilon,
\end{split}
\]
where we used the assumption $A_kx_k\in K$ in the second line and \eqref{temp02} in the last. Taking $\varepsilon \to 0$ we thus have that $l(Ax) \leq l_0$, and the claim follows, since every closed convex set in $Y$ equals to the intersection of all its supporting hyperplanes (see \cite[Corollary~21.8]{JCR_book}, for example).
\end{proof}

\bibliographystyle{plain}
\bibliography{literature}

\end{document}